\documentclass[12pt]{article}
\usepackage{amssymb,amsmath, amsthm, amscd,ifthen}
\usepackage[T2A]{fontenc}
\usepackage[cp1251]{inputenc}
\usepackage[russian,english]{babel}
\usepackage[dvips]{graphicx}
\usepackage{indentfirst}
\usepackage{bm}
\newcommand{\N}{{\mathbb N}}
\newtheorem*{thm*}{Theorem}
\newcommand{\ff}{{\mathcal F}}

\newtheorem*{cla*}{Claim}

\newtheorem{thm}{Theorem}

\newtheorem*{gypo}{Conjecture}

\newtheorem{prop}[thm]{Proposition}
\newtheorem{lem}[thm]{Lemma}

\date{}
\title{Random Kneser graphs and hypergraphs}
\begin{document}
\author{Andrey Kupavskii\footnote{Laboratory of Advanced Combinatorics and Network Applications at Moscow Institute of Physics and Technology, University of Oxford; Email: {\tt kupavskii@ya.ru} \ \ Research supported by the grant RNF~16-11-10014.}}

\maketitle

\begin{abstract} The Kneser graph $KG_{n,k}$ is the graph whose vertices are the $k$-element subsets of $[n],$ with two vertices adjacent if and only if the corresponding sets are disjoint. A famous result due to Lov\'asz states that the chromatic number of $KG_{n,k}$ is equal to $n-2k+2$. In this paper we discuss the chromatic number of random Kneser graphs and hypergraphs. It was studied in two recent papers, one due to Kupavskii, who proposed the problem and studied  the graph case, and the more recent one due to Alishahi and Hajiabolhassan.
The authors of the latter paper had extended the result of Kupavskii to the case of general Kneser hypergraphs. Moreover, they have improved the bounds of Kupavskii in the graph case for many values of parameters.

In the present paper we present a purely combinatorial approach to the problem based on blow-ups of graphs, which gives much better bounds on the chromatic number of random Kneser and Schrijver graphs and Kneser hypergraphs. This allows us to improve all known results on the topic. The most interesting improvements are obtained in the case of $r$-uniform Kneser hypergraphs with $r\ge 3$, where we managed to replace certain polynomial dependencies of the parameters by the logarithmic ones.
\end{abstract}

\section{Introduction}
Kneser graphs and hypergraphs are very popular and well-studied objects in combinatorics. Fix some positive integers $n,k,r$, where $r\ge 2$. The set of vertices of the Kneser $r$-graph $KG_{n,k}^r$ is the set of all $k$-element subsets of $[n]$, denoted by ${[n]\choose k}$. The set of edges of $KG_{n,k}^r $ consists of all $r$-tuples of pairwise disjoint subsets. Thus, $KG_{n,k}^r$ is non-empty only if   $n\ge kr$. Substituting $k = 1$ in the definition gives the complete $r$-graph on $n$ vertices. When dealing with the graph case $r=2$, we omit the superscript in the notation of Kneser graphs. For a hypergraph $\mathcal H$ we denote by $\chi(\mathcal H)$ its \textit{chromatic number}, that is, the minimum number $\chi$ such that there exists a coloring of vertices $\mathcal H$ into $\chi$ colors that leaves no edge of $\mathcal H$ monochromatic.

 Naturally, Kneser graphs were studied first. They earned their name from M. Kneser, who investigated them in the paper \cite{Knez}. He showed that $\chi(KG_{n,k})\le n-2k+2$ and conjectured that this bound is tight. This conjecture (or rather its resolution) played a very important role in combinatorics. It was confirmed by L. Lov\'asz \cite{Lova}, who, in order to resolve it, introduced tools from algebraic topology to combinatorics.

 Once \cite{Lova} appeared, there was a burst of activity around Kneser graphs. I. B\' ar\'any \cite{Bar} gave an elegant alternative proof of Lov\'asz' result, and several authors studied the chromatic number of Kneser (disjointness) graphs of arbitrary set systems. In particular, there were results due to V. Dol'nikov \cite{Dol} and A. Schrijver \cite{Sch}.

 In \cite{Sch}, Schrijver studied induced subgraphs $SG_{n,k}$ of $KG_{n,k}$ constructed on the family of all $k$-element \textit{stable sets} of the cycle $C_n$. In other words, the underlying family contains all $k$-element sets that do not have two cyclically consecutive elements of $[n]$. Schrijver noticed that a slight modification of B\' ar\'any's proof yields a stronger statement: $\chi(SG_{n,k})=n-2k+2.$ In the harder part of the paper, he also showed that $SG_{n,k}$ is a vertex-critical subgraph of $KG_{n,k}$, i.e. that any proper induced subgraph of $SG_{n,k}$ has strictly smaller chromatic number.

 A coloring of $KG_{n,k}$ in $n-2k+2$ colors is easy to obtain: for each $1\le i\le n-2k+1$ color the sets with minimum element $i$ into color $i$, and color the remaining sets, forming the family ${[n-2k+2,n]\choose k}$, into color $0$. A similar coloring for $KG^r_{n,k}$ gives the upper bound $\chi(KG^r_{n,k})\le \big\lceil\frac{n-r(k-1)}{r-1}\big\rceil$: for $1\le i\le n-kr+1$ color the sets with minimum element in  $[(r-1)(j-1)+1,(r-1)j]$ into color $j$, and color the sets from ${[n-rk+2,n]\choose k}$ into color $0$. P. Erd\H os \cite{Erd1} conjectured that this bound is sharp for all $r\ge 2$. After some partial progress it was confirmed in full generality by N. Alon, P. Frankl, and L. Lov\'asz \cite{AFL}. The proof again used topological tools.

Generalizing both the result of Dolnikov \cite{Dol} and Alon, Frankl, and Lovasz \cite{AFL}, I. K\v r\'\i\v z \cite{Kr1}, \cite{Kr2} obtained the bound on the chromatic number of Kneser hypergraphs of general set families. Later, an elegant alternative proof was obtained by J. Matou\v sek \cite{Mat2}, and some more general results with combinatorial proofs were obtained by G. Ziegler \cite{Z}. We also refer to the amazing book written by J. Matou\v sek on the subject \cite{Mat}.

In \cite{ADL}, N. Alon, L. Drewnowski and T. \L uczak  applied results on colorings of Kneser-type hypergraphs for constructing certain ideals in $\N$. The hypergraphs they considered are called \textit{$s$-stable Kneser hypergraphs}, and they may be seen as a generalization of Schrijver graphs. The underlying set family that defines the $s$-stable Kneser $r$-hypergraph $KG_{n,k}^{r,\ s-\text{stable}}$ consists of all $k$-element subsets $\{i_1,\ldots, i_k\}\subset [n]$, whose consecutive elements are sufficiently far apart: if $1\le i_1<\ldots<i_k\le n$, then for any $j=0,\ldots, k-1$ the elements satisfy $i_{j+1}-i_{j}\ge s$, as well as $i_1+n-i_k\ge s$. The case $r=s=2$ corresponds to Schrijver graphs.

In their paper, Alon, Drewnowski and \L uczak proved and applied the following result: $\chi(KG_{n,k}^{r,\ r-\text{stable}}) = \chi(KG_{n,k}^r)$ for $r=2^t,\ t\in\N$. They have also stated explicitly the conjecture tracing back to Ziegler's paper \cite{Z}, which says that the same equality holds for any $r$. We are going to use this result of \cite{ADL}, and we note that it would have improved some of the bounds in this paper, shall the conjecture be verified. A more general conjecture was made by F. Meunier \cite{M}: $\chi(KG_{n,k}^{r,\ s-\text{stable}}) = \big\lceil\frac{n-s(k-1)}{r-1}\big\rceil$ for any $s\ge r$. It was verified in some cases, but is still wide open in general.\\

In fact, Kneser was not the first to ask a question concerning Kneser graphs.  P. Erd\H os, C. Ko, and R. Rado \cite{EKR} proved that the size of the largest family of $k$-element subsets of $[n]$ with no two disjoint sets is at most ${n-1\choose k-1}$, provided that $n\ge 2k$. In terms of $KG_{n,k}$, they determined its independence number, that is, the maximum size of a subset of vertices not containing an edge of the graph. Later, Erd\H os \cite{E} asked a more general question: what is the size of the largest family of $k$-element subsets of $[n]$ with no $r$ pairwise disjoint sets? This is obviously a question about the independence number of $KG_{n,k}^r$, and, unlike the question on the chromatic number, it does not have a complete solution yet. However, the question was resolved for a wide range of parameters by P. Frankl \cite{F4} and by Frankl and the author \cite{FK13}. For some recent progress on the subject see \cite{FK6}, \cite{FK7}.\\

 An $r$-uniform Kneser hypergraph of any $k$-uniform set system is an induced subgraph of $KG_{n,k}^r$, and thus the results on the chromatic number of induced subgraphs of $KG_{n,k}^r$ belong to the class of results discussed above. Instead of restricting to a subset of vertices, in this paper we study, what happens if we restrict to a subset of edges. The most natural model to study is the binomial model of  a random hypergraph. For a hypergraph $\mathcal H$ and a real number $p$, $0<p<1$, define the {\it random hypergraph} $\mathcal H(p)$ as follows: the set of vertices of $\mathcal H(p)$ coincides with that of $\mathcal H$, and the set of edges of $\mathcal H(p)$ is a subset of the set of edges of $\mathcal H$, with each edge from $\mathcal H$ taken independently and with probability $p$. The results on random graphs and hypergraphs, roughly speaking, tell us how does a \textit{typical} subgraph of a given (hyper)graph that contains a $p$-fraction of edges behave with respect to a given property. Theory of random graphs and hypergraphs is very rich in both results and open problems, and by no means we are going to give an overview of the field in this paper. We refer the reader to the books \cite{AS}, \cite{Boll} for some classical results on the subject.

 One class of questions that is particularly relevant for this paper deals with transference results. In general, we speak of transference if a certain combinatorial result holds with no changes in the random setting. One example of such theorem is due to B. Bollob\'as, B. Narayanan and A. Raigorodskii \cite{BNR}. They studied the size of maximal independent sets in $KG_{n,k}(p)$, and showed that for a wide range of parameters the independence number of $KG_{n,k}(p)$ is \textit{exactly} the same as that of $KG_{n,k}$, given by the Erd\H os--Ko--Rado theorem. Later on, their result was further strengthened by J. Balogh, B. Bollob\'as, and B. Narayanan \cite{BBN}, S. Das and T. Tran \cite{DT} and P. Devlin and J. Kahn \cite{DK}.

In the paper \cite{Kup}, the author studied the behaviour of the chromatic number of $KG_{n,k}(p)$ and $SG_{n,k}(p)$, showing that, compared to $\chi(KG_{n,k})$, it does change at most by a small additive term in a very wide range of parameters.

Random subgraphs of more general graphs $K(n,k,l)$ were investigated by L. Bogolyubskiy, A. Gusev, M. Pyaderkin and A. Raigorodskii in \cite{BGPR,BGPR2}. The vertices of $K(n,k,l)$ are the $k$-element subsets of $[n]$, with two vertices adjacent if the corresponding sets intersect in exactly $l$ elements. In \cite{BKR,BGPR,BGPR2,Pyad, PR, Rai} the authors obtained several results concerning the independence number and the chromatic number of $K(n,k,l)(p)$ and related results. \\


In the recent paper \cite{AH}, which motivated in part the present paper, M. Alishahi and H. Hajiabolhassan generalized the results of \cite{Kup} to the case of Kneser hypergraphs of arbitrary set systems. They have also strengthened the results of \cite{Kup} in the graph case.  The proofs of Alishahi and Hajiabolhassan are quite technically involved and not easy to follow. They use topological methods to obtain their results.

In this paper we describe a purely combinatorial approach to the problem, which allows us to significantly improve all previously known bounds on the chromatic numbers in the most interesting cases: for random subgraphs of (complete) Kneser and Schrijver graphs and Kneser hypergraphs. Our method may be extended to more general classes of Kneser hypergraphs, which we discuss in Section~\ref{sec5}. This does not, however, cover all generalized Kneser hypergraphs, so the result of Alishahi and Hajabolhassan remains best known in some cases.

\section{The old and the new bounds}\label{sec2}
In this section we discuss both the old and the new quantitative bounds on the chromatic number of Kneser and Schrijver graphs and hypergraphs. We do not state the bounds in full generality as they depend on too many parameters and thus are very difficult to interpret. We preferred clarity to generality, and instead focused on several most interesting (in our opinion) cases. These cases were also discussed in \cite{Kup} and \cite{AH}, so we can compare the results. The bounds in their full generality appear in the latter sections.

For the rest of the section, we assume that $r\ge 2, p>0$ are fixed. Note that in the case $r=2$ we formulate our results for Schrijver graphs $ SG_{n,k}(p)$. The same bounds hold for Kneser graphs, since  Schrijver graphs  are subgraphs of Kneser graphs.

We henceforth use the notation $f(n)\gg g(n)$ in a slightly unconventional way. This inequality should be read as: there exists a sufficiently large constant $C$, where $C=C(r,p,k)$ if $k$ is fixed, or $C=C(r,p,l)$ if $l$ is fixed, such that $f(n)\ge Cg(n)$ for all sufficiently large $n$. All logarithms with unspecified base have base $e$. All statements below hold asymptotically almost surely (a.a.s.) for $n\to \infty$, and so we omit writing a.a.s. in most statements for brevity.

Returning to the results on colorings, the author \cite{Kup} proved that a.a.s.
\begin{align}\label{eq01}(\text{\bf$\pmb{l=1}$) } \ \ \chi(SG_{n,k})(p)&\ge \chi(KG_{n,k+1}) = \chi(KG_{n,k})-2 \ \ \ \textit{if }n-2k \ll \sqrt n;\\
\label{eq02}\text{\bf(fixed $\pmb l$)} \ \ \chi(SG_{n,k})(p)&\ge \chi(KG_{n,k+l}) =\chi(KG_{n,k})-2l\ \ \ \textit{if }l\text{ is fixed}\ \text{and} \ k \gg n^{\frac 3{2l}};\\
\label{eq03}\text{\bf (fixed $\pmb k$)} \ \ \chi(SG_{n,k})(p)&\ge \chi(KG_{n,k+l}) =\chi(KG_{n,k})-2l\ \ \ \textit{if }k\text{ is fixed}\ \text{and}  \ l \gg n^{\frac 3{2k}}.
 \end{align}
Note that in the last two $k$ and $l$ are simply interchanged in the conditions needed for the inequality to hold. The following bounds were proven by Alishahi and Hajiabolhassan \cite{AH}:
\begin{align}\label{eq04}\text{\bf ($\pmb{l=1}$) } \ \ \chi(KG_{n,k}^r)(p)&\ge \chi(KG^r_{n,k+1})\ \ \ \textit{if }n-rk \ll n^{\frac{r-1}r};\\
\label{eq05}\text{\bf (fixed $\pmb l$)} \ \ \chi(KG^r_{n,k})(p)&\ge \chi(KG^r_{n,k+l}) \ \ \ \textit{if }l\text{ is fixed}\ \text{and}  \ k \gg n^{\frac {r}{lr-1}}\log^{\frac 1{lr-1}} n.
 \end{align}
We do not express $\chi(KG^r_{n,k+l})$ in terms of $\chi(KG^r_{n,k})$, since the formulas are much uglier in the hypergraph case. Note that for $r=2$ the bound (\ref{eq04}) coincides with (\ref{eq01}),  while (\ref{eq05}) improves on (\ref{eq02}).\\

In this paper we prove the following bounds.
\begin{thm}\label{thm0} Let $p\in (0,1)$ and $r\in \N, r\ge 2$, be fixed. \\
If $r=2^q$ for some $q\in \mathbb N$ then a.a.s.
\begin{align}\label{eq08}\text{\bf ($\pmb{l=1,\ r=2^q}$) } \ \ \chi(KG_{n,k}^r)(p)&\ge \chi(KG_{n,k+1}^r) \ \ \ \textit{if }n-rk \ll n^{r/(r+1)}\log^{-1/(r+1)}n.
\end{align}
If $r=2$ then a.a.s.
\begin{align}
\label{eq010}\text{\bf (fixed $\pmb {l,\ r=2}$)} \ \ \ \chi(KG_{n,k})(p)&\ge \chi(KG_{n,k+l}) \ \ \ \textit{if }l\text{ is fixed and} \ k \gg (n\log n)^{1/l};\ \ \ \\
\label{eq0101}\text{\bf (fixed $\pmb {k,\ r=2}$)} \ \ \ \chi(KG_{n,k})(p)&\ge \chi(KG_{n,k+l}) \ \ \ \textit{if }k\text{ is fixed and} \ l \gg (n\log n)^{1/k}.\ \
\end{align}
If $r=3$ then a.a.s.
\begin{align}
\label{eq011}\text{\bf (fixed $\pmb {l,\ r=3}$)} \ \ \chi(KG^3_{n,k})(p)&\ge \chi(KG^3_{n,k+l}) \ \ \ \textit{if }l\text{ is fixed and} \ k \gg \log^{1/(3l-4)} n;\\
\label{eq0111}\text{\bf (fixed $\pmb {k,\ r=3}$)} \ \ \chi(KG^3_{n,k})(p)&\ge \chi(KG^3_{n,k+l}) \ \ \ \textit{if }k\text{ is fixed and} \ l \gg \log^{2/(6k-11)} n.
\end{align}
If $r>3$ then a.a.s.
\begin{align}
\label{eq012}\text{\bf (fixed $\pmb {l,\ r>3}$)} \ \chi(KG^r_{n,k})(p)&\ge \chi(KG^r_{n,k+l}) \  \textit{if }l\text{ is fixed and} \ k \gg \log^{\frac 1{r(l-2)-1}} n;\\
\label{eq0121}\text{\bf (fixed $\pmb {k,\ r>3}$)} \ \chi(KG^r_{n,k})(p)&\ge \chi(KG^r_{n,k+l}) \  \textit{if }k\text{ is fixed and} \ l \gg \log^{\frac 1{r(k-1)-\frac{2r-1}{r-1}}} n.
 \end{align}
\end{thm}
We remark that the bounds with $r=2$ stated in the theorem hold for Schrijver graphs and, more generally, for $r$-stable $r$-uniform Kneser hypergraphs, when $r=2^t$ for some $t\in \N$.

In the graph case, we see that (\ref{eq08}), (\ref{eq010}), and (\ref{eq0101}) improve on (\ref{eq01}), (\ref{eq02}) and (\ref{eq03}), respectively.

Our most interesting results are (\ref{eq011})--(\ref{eq0121}). They are much stronger than (\ref{eq05}) and guarantee that the chromatic number of $KG_{n,k}^r(p)$ drops by no more than a small additive term  for already for polylogarithmic $k$ (this was known before for polynomial $k$).

One question that arises in this context is what makes the case $r=2$ so different from the case $r>2$? Can one obtain a bound similar to (\ref{eq011})--(\ref{eq0121}) for the case $r=2$?

In the next section we present the general approach to the problem, and obtain  inequality (\ref{eq08}). The approach, which is more adapted to our particular problem, is presented in Section \ref{sec4}. The rest of the inequalities from Theorem \ref{thm0} are obtained there. In Section~\ref{sec50} we prove some simple upper bounds. In Section \ref{sec5} we discuss some directions for further research.

\section{Basic approach}\label{sec3}
In this section we discuss the general method, proposed to us by N. Alon, along with some of its corollaries to the case of Kneser and Schrijver graphs and hypergraphs. We prove (\ref{eq08}) and reprove (\ref{eq04}) in this section.
\subsection{Coloring random subgraphs of blow-ups of hypergraphs}
We start with the following abstract theorem on hypergraph colorings, preceded by the definition of the class of hypergraphs in question.
For an $r$-uniform hypergraph $\mathcal H=(V, E)$ and a positive integer number $m$ consider the {\it $m$-blow-up} $\mathcal H[m]$ of $\mathcal H$: $\mathcal H[m]=(V', E'),$ where $V':=V\times[m]$, and $E':=\bigl\{\{v_1\times i_1,\ldots, v_r\times i_r\}: \{v_1,\ldots,v_r\}\in E, i_1,\ldots, i_r\in [m]\bigr\}$. Informally speaking, we replace each vertex of the original hypergraph with an $m$-tuple, and each edge with a complete $r$-partite hypergraph with $m$ vertices in each part.

We denote by $\mathcal A(\mathcal H,m)$ the class of hypergraphs that can be obtained from $\mathcal H[m]$ by identifying some vertices that do not belong to the same edge and do not arise from the same vertex of $H$. Formally, consider the class $\mathcal F$ of functions $f: V'\to [n]$ for some $n$, such that:\vskip+0.1cm\noindent
 1. For any  $v\in V$ and $i\ne j$ we have $ f(v\times i)\ne f(v\times j).$\\
  2. For any $e\in E'$ and $v_1,v_2\in e$ we have  $f(v_1)\ne f(v_2).$\\
   3. The function $f$ is surjective.\vskip+0.1cm

\noindent Then the class of hypergraphs $\mathcal A(H,m)$ is defined as follows: $$\mathcal A(H,m) := \Bigl\{\bigl(f(V'),E_f\bigr): f\in \mathcal F, E_f := \bigl\{\{f(v_1),\ldots,f(v_r)\}: \{v_1,\ldots,v_r\}\in E'\bigr\}\Bigr\}.$$

We denote by $K^r[m]$ the complete $r$-partite $r$-uniform hypergraph with parts of size $m$. For any $0<p<1$ and a hypergraph $\mathcal H$ we define the random hypergraph $\mathcal  H(p)$, which has the same set of vertices and in which each edge from $G$ is taken independently and with probability $p$.

\begin{thm}\label{thm1} Let $\mathcal H=(V,E)$ be an $r$-uniform hypergraph with $\chi(\mathcal H) = d+1$. Fix a number $m\in
\mathbb N$ and consider a hypergraph $\mathcal G\in \mathcal A(\mathcal H,m)$. Then for any coloring of $\mathcal G$ into $d$ colors there is a subhypergraph $K^r[\lceil\frac md\rceil]\subset G$ with all vertices colored in the same color.

Moreover, for any $0<p<1$ we have \begin{equation}\label{eq2}\Pr[\chi(\mathcal G(p))\le d]\le |E|{m\choose \lceil m/d\rceil}^r(1-p)^{\lceil m/d\rceil^r}.\end{equation}
\end{thm}

\begin{proof} Consider a coloring of $\mathcal G$ into $d$ colors. We construct a certain coloring of $\mathcal H$ based on the coloring of $\mathcal G$. For each vertex  $v\in \mathcal H$ take its blow-up $\{v_1,\ldots, v_m\}$ in $\mathcal G$ and color $v$ in the most popular color among $v_1,\ldots,v_m$. It is clear that at least $\lceil\frac md\rceil$ of $v_i$'s are colored in this color.

Since $\chi(\mathcal H)>d$, there is a monochromatic edge of color $\kappa$ in this coloring. In $\mathcal G$ this edge corresponds to an $r$-uniform $r$-partite subhypergraph $K^r[m]$. Choosing out of each part the vertices colored in color $\kappa$, we get the desired subhypergraph.

In view of the argument above, the event ``$\chi(\mathcal G(p))\le d$'' may occur only if one of the subhypergraphs $K^r[\lceil\frac md\rceil]$ of the type described above is empty in the random hypergraph $\mathcal G(p)$. The number of such subhypergraphs is bounded from above by $|E|{m\choose \lceil m/d\rceil}^r$, while the probability for each to be empty in $\mathcal G(p)$ is $(1-p)^{\lceil m/d\rceil^r}$. Thus, inequality (\ref{eq2}) follows by applying the union bound.
\end{proof}

\subsection{Numerical Corollaries for Kneser hypergraphs}
For $n\ge (k+l)r$  the hypergraph $KG_{n,k}^r$ belongs to the family $\mathcal A(KG_{k+l}^r,{k+l\choose k})$. Indeed, to each vertex $S\in {[n]\choose k+l}$ of $KG_{n,k+l}^r$ we correspond the family of subsets ${S\choose k}$. Put $d:=\chi(KG_{n,k+l}^r)-1$ and
\begin{equation}\label{eq31}t:=\Bigl\lceil{k+l\choose k}/d\Bigr\rceil.\end{equation}
The following lemma gives the first (but not the strongest) general bound on the chromatic number of random Kneser hypergraphs.
\begin{lem}\label{coro1}
   For $n\ge (k+l)r$ we a.a.s. have $\chi(KG^r_{n,k}(p))\ge d+1$ if
\begin{equation}\label{eq3} 3r\Bigl((k+l)\log \frac nk +t\log d\Bigr)-pt^r\to -\infty\end{equation}
\end{lem}
\begin{proof}
Remark that the number of edges in $KG^r_{n,k+l}$ is at most $$ {n\choose k+l}^r\le \Bigl(\frac{ne}k\Bigr)^{(k+l)r}.$$
Therefore, applying the bound (\ref{eq2}), we get that
$$\Pr[\chi(G(p))\le d]\le |E(KG^r_{n,k+l})|{{k+l\choose k}\choose t}^r(1-p)^{t^r}\le \Bigl(\frac{ne}k\Bigr)^{(k+l)r}(ed)^{rt}e^{-pt^r}\le $$
$$\exp\Bigl[(k+l)r\Bigl(1+\log \frac nk\Bigr)+rt(1+\log d)-pt^r\Bigr]\le \exp\Bigl[3r\Bigl((k+l)\log \frac nk+t\log d)\Bigr)-pt^r\Bigr].$$
The last expression tends to 0 by (\ref{eq3}), which concludes the proof of the lemma.
\end{proof}

The condition in \eqref{eq3} is satisfied for fixed $p, r$ and $l$, provided $t^{r-1}\gg \log d$ and $t^r\gg k\log \frac nk$. Substituting the value of $t$ and doing some tedious calculations (see the proofs of the corollaries in \cite{Kup} for more details), we get that the following hold a.a.s.:
\begin{align}\label{eq41}\text{(\bf$\pmb{l=1}$) } \ \ \chi(KG_{n,k}^r)(p)&\ge \chi(KG^r_{n,k+1})\ \ \ \textit{if }n-rk \ll n^{\frac{r-1}r};\\
\label{eq42}\text{\bf (fixed $\pmb l$)} \ \ \chi(KG^r_{n,k})(p)&\ge \chi(KG^r_{n,k+l}) \ \ \ \textit{if }k \gg n^{\frac {r}{lr-1}}\log^{\frac 1{lr-1}}\ \text{ and }l\text{ is fixed.}
 \end{align}


The first bound is the same as (\ref{eq01}) and (\ref{eq04}), while the second one is the same as the bound (\ref{eq05}). However, (\ref{eq42}) is still a long way from the latter bounds in Theorem \ref{thm1}. \\

The way to improve the bound for $l=1, r=2^q$ is to work with $r$-stable Kneser hypergraphs. As in the case of (complete) Kneser hypergraphs, we have $KG_{n,k}^{r,\ r-\text{stable}}\in \mathcal A(KG_{n,k+l}^{r,\ r-\text{stable}},{k+l\choose k})$ and, if $n-rk=o(n)$, it has fewer vertices and edges than $KG_{n,k}^r$.


\begin{prop}\label{prop1} The number of vertices in $KG_{n,k}^{r,\ r-\text{stable}}$ is at most ${n-(r-1)(k-1)\choose k}$.
\end{prop}
 \begin{proof} The vertices of $KG_{n,k}^{r,\ r-\text{stable}}$ are the $k$-subsets $\{i_1,\ldots, i_k\}$ of $[n]$ that satisfy $i_{j+1}-i_{j}\ge r$ for each $j=0,\ldots, k-1$, as well as $i_1+n-i_k\ge r$, provided that $1\le i_1<\ldots<i_k\le n$.  Let us count the number $f(n,k)$ of $k$-sets satisfying all these restrictions except $i_1+n-i_k\ge r$. This number will clearly be an upper bound for $|V(KG_{n,k}^{r,\ r-\text{stable}})|$.

 It is easy to see that this quantity satisfies the following recursive formula: $f(n,k) = f(n-1,k)+f(n-r,k-1)$, as well as the condition $f(r(k-1)+1,k)=1$. The function ${n-(r-1)(k-1)\choose k}$ satisfies both the recursive formula and the initial condition.\end{proof}

Put $d:=\chi(KG_{n,k+1}^{r,\ r-\text{stable}})-1$. From Proposition \ref{prop1} we get that the number of edges in $KG_{n,k+1}^{r,\ r-\text{stable}}$ is at most
$${n-(r-1)(k-1)\choose k}^r = {k+O(d)\choose k}^r = {O(n)\choose O(d)}^r = e^{O(d\log \frac nd)}.$$  Thus, instead of (\ref{eq3}) it is sufficient to show that (recall that $p$ and $r$ are fixed)
\begin{equation}\label{eq43} d\log \frac nd +t\log d\ll t^r.\end{equation}

We have $t = \Theta(\frac n d)$ (see \eqref{eq31}) and $d=O (n-rk)$. Doing some routine calculations again, we get that a.a.s.
\begin{equation*}\label{eq44}\text{\bf($\pmb{l=1}$) } \ \ \chi(KG_{n,k}^{r,\ r-\text{stable}})(p)\ge \chi(KG_{n,k+1}^{r,\ r-\text{stable}}) \ \ \ \textit{if }n-rk \ll n^{r/(r+1)}\log^{-1/(r+1)}n.\end{equation*}
Since for $q\in \mathbb N$ and $r=2^q$ we have $\chi(KG_{n,k+1}^{r,\ r-\text{stable}}) = \chi(KG_{n,k+1}^r)$, we get (\ref{eq08}).

\section{The approach refined}\label{sec4}
The crucial step in the proof of Threorem \ref{thm1} is to get a monochromatic edge of $KG_{n,k+l}^r$, induced by the coloring of $KG_{n,k}^r$. The main limitation of the method from the previous section is related to this step. We have to assume that (in the worst case) among the vertices of the $m$-blow-up of the monochromatic edge all colors are represented in approximately the same proportion. This is why we can only guarantee the majority color class to have  size at least $\frac md$. On the other hand, in order to get a good bound on the probability, we need to work with color classes of growing size. Therefore, the approach from the previous section is bound to work only for $m\gg d$, or, in terms of Kneser hypergraphs, for ${k+l\choose k}\gg \chi(KG_{n,k+l}^r)$.


In this section we are going to  partially overcome the aforementioned difficulty. We assume that $n\gg k+l$ and that $r\ge 2$ is fixed for the rest of the section. Put $$d:=\chi(KG_{n,k+l}^r)-1.$$
Note that $d=n-2k-2l+1$ for $r=2$. Fix a coloring $\kappa$ of ${[n]\choose k}$ in $d$ colors. For each subset $S\subset [n]$ of size at least $k$, define the color of $S$ to be the most popular color among its subsets. We have thus defined the coloring $\kappa'$ of $KG_{n,k'}^r$ for all $k'\ge k$.
Put $u := \lfloor\log_2 {\frac{n}{k+l^{\beta}}}\rfloor$ and consider the following sequence of numbers
\begin{equation}\label{eq51} q_0 := k+l, \ \ \ \ q_i:=\lceil 2^i(k+l^{\beta})\rceil, \ \ \ \ \text{where } \ \ \ i = 1,\ldots,u\ \ \ \text{and } \  \beta = \begin{cases}1\ \ \ \ \text{if} \ \ r=2\\ \frac {r}{r-1}\ \text{if} \ \ r\ge 3. \end{cases}\end{equation}
Note that, by definition, no $q_i$ is bigger than $n$ and that $q_{i}\le 2l^{\beta-1}q_{i-1}$ for each $i\in[u]$. The numbers $q_i$ will play the role of the sizes of subsets on which we construct our Kneser hypergraphs. Next, for each $i = 0,\ldots, u$, define the following two numbers:
\begin{equation}\label{eq52} t_i:=\Big\lceil\frac{{q_i\choose k}}d\Big\rceil,\ \ \ \ \ z_i:=\Big\lceil\frac{{q_i\choose k}} {2sq_i}\Big\rceil,\ \ \text { where }\ \ s := \begin{cases}(2r+1)l^{\beta-1}\ \ \ \, \text{if }\ r= 2,3;\\
(2r+1)l^{\beta-1}k \ \ \text{if }\ r>3.\end{cases}\end{equation}
Note that $t_0$ is equal to $t$ from the previous section. Both $t_i$ and $z_i$ will play the roles of the sizes of popular colors among the $k$-subsets of a certain $q_i$-element set.

The following lemma is central for this section.

\begin{lem}\label{lem1} For any coloring $\kappa$ of $KG_{n,k}^r$ in $d$ colors there is a color $\alpha$ for which one of the following holds.
\begin{itemize}
\item[(i)] There exists $i\in\{0,\ldots, u\}$ and $r$ pairwise disjoint subsets $A_1,\ldots, A_r \in {[n]\choose q_i}$, such that at least $z_i$ subsets from each ${A_j\choose k},$ $j=1,\ldots,r$, are colored in $\alpha$.
\item[(ii)] There exists $i\in\{0,\ldots, u-1\}$ and $r$ pairwise disjoint subsets $A_1\in {n\choose q_i}$, $A_2,\ldots, A_r \in {[n]\choose q_{i+1}}$, such that at least $t_i$ subsets from ${A_1\choose k}$ and $z_{i+1}$ subsets from each ${A_j\choose k},$ $j=1,\ldots,r$, are colored in $\alpha$.
\end{itemize}
\end{lem}

\begin{proof}
 We start by analyzing the colorings of $KG_{n,k}^r$ into $d$ colors. For each $i=0,\ldots, u$ associate with each set $A\in {[n]\choose q_i}$ the set $X_A$ of the colors that are used at least $\frac {t_i} 2$ times in the coloring of ${A\choose k}$ (recall the definitions (\ref{eq51}), (\ref{eq52})). Note that at least a half of the vertices of ${A\choose k}$ is colored by colors from $X_A$.  We have two possibilities for a given $i$: either for each set $A\in {[n]\choose q_i}$ there is a color that is used for $z_i$ vertices on ${A\choose k}$, or there is a set $A\in {[n]\choose q_i}$ such that $|X_A|>sq_i$. This is obvious in case $z_i\le t_i$. If $z_i>t_i$ (which is typically the case), then the negations of both statements imply that in  $X_A$ there are less than $sq_i$ colors, each of cardinality at most $z_i-1$.  On the other hand, the colors from $X_A$ are used for at least $\frac 12 {q_i\choose k}$ $k$-sets in $A$, but $\frac 12{q_i\choose k}>(z_i-1)sq_i$. This is a contradiction.

If for each $q_i$-element set $A$ there is a color in $X_A$ used at least $z_i$ times, then, arguing as in the proof of Theorem \ref{thm1}, we conclude that (i) from the lemma takes place.

If not, then fix \textit{the largest} index $i$, for which there exist a $q_i$-element set $A$ with $|X_A|>sq_i$, and choose such $A$. Clearly, $i\le u-1$, otherwise we have more than $d$ colors in $X_A$. Put $Y :=[n]\setminus A$ and consider the majority coloring $\kappa'$ of the $q_{i+1}$-element subsets of $Y$. We denote by  $KG_{Y,q_{i+1}}^r$ the Kneser hypergraph induced on ${Y\choose q_{i+1}}$.

We claim that at least one of the two holds. Either \textbf{A} $\kappa'$ is not proper, and we again conclude that (i) holds (remark that for any $q_{i+1}$-set $B$ the set $X_B$ has a color class of size $z_{i+1}$ by the choice of $i$), or \textbf{B} in $KG_{Y,q_{i+1}}^r$ there is a color from $X_A$ used for an $(r-1)$-tuple of pairwise disjoint $q_{i+1}$-element sets. If the second possibility takes place, then we obtain (ii) from the lemma.
Thus, we are left to show that one of the two possibilities claimed in this paragraph takes place.\vskip+0.1cm

Assume that neither of the two options \textbf{A, B} takes place.   In the case $r=2$, this simply means that $KG_{Y,q_{i+1}}$ is properly colored in less than $d-sq_{i}$ colors (the negation of \textbf{B} implies that the colors from $X_A$ are not used in $\kappa'$). Remark that $\beta=1$ for $r=2$ and thus $q_{i+1}=2q_{i}$ for each $i=0,\ldots, u-1$.  Therefore, $n-2k-2l+1-sq_i =d-sq_i>\chi(KG_{Y,q_{i+1}})= (n-q_i)-2q_{i+1}+2,$ which is equivalent to $(s-5)q_i+2(k+l)+1<0$. But, by the definition (\ref{eq52}), $s = 5$ when $r=2$. We arrive at a contradiction, which concludes the proof for $r=2$.

In the case $r \ge 3$ the negation of \textbf{A, B} means that  $KG_{Y,q_{i+1}}^r$ is properly colored, and that for each color $\alpha\in X_A$ the collection of sets colored in $\alpha$ contains at most $r-2$ pairwise disjoint sets.

If $r=3$, then construct a new coloring of $KG_{Y,q_{i+1}}^3$ by arbitrarily grouping the colors from $X_A$ into pairs and replacing each pair by a single new color. Remark that the new coloring uses at most $d-\frac {sq_i}2$ colors and is still proper, since in any  newly formed color there are no more than two pairwise disjoint sets. Therefore, $\chi(KG^3_{n,k+l})-1-\frac {sq_i}2 \ge \chi(KG^3_{Y,q_{i+1}})$.

Recall that $q_{i+1}\le 2l^{\beta-1} q_i$.  For any $r\ge 2$, we have \begin{small}\begin{multline}\label{eq56}\chi(KG^r_{n,k+l})- \chi(KG^r_{Y,q_{i+1}}) = \\ \Bigg\lceil\frac{n-(k+l-1)r}{r-1}\Bigg\rceil - \Bigg\lceil\frac{n-q_i-(q_{i+1}-1)r}{r-1}\Bigg\rceil<\frac{rq_{i+1}+q_i}{r-1} \le \frac{(2r+1)l^{\beta-1}q_{i}}{r-1}.\end{multline}\end{small}
Thus, for $r=3$ we conclude that $\frac {sq_i}2<\frac{(2r+1)l^{\beta-1}q_i}2$, which contradicts (\ref{eq52}).

Finally, consider the case $r>3$. We again construct a new coloring of $KG_{Y,q_{i+1}}^r$ that uses fewer colors than $\kappa'$, using the following procedure. First, split the colors from $X_A$ into groups of size $k(r-2)+1$ (with one remaining group of potentially smaller size). In each group  choose one color $\alpha$ and split vertices (sets) from $KG_{Y,q_{i+1}}$ of color $\alpha$ into $k(r-2)$ groups of pairwise intersecting sets. To obtain such a split, take the largest family  of pairwise disjoint sets in $KG_{Y,q_{i+1}}$ colored in $\alpha$. It has size most $r-2$, and thus it covers the set $U\subset Y$ of cardinality at most $(r-2)k$. Each other set of color $\alpha$ in $Y$ must intersect $U$. We then split all sets of color $\alpha$ into families $\mathcal K_j, j=1,\ldots, |U|$ of sets containing the $j$-th element of $U$.

Next, we adjoin each of the families $\mathcal K_j$ to one of the remaining $k(r-2)$ colors in the group. We get a proper coloring since none of the newly formed colors contain more than $r-1$ pairwise disjoint sets. The number of colors used in the new coloring is less than $d-\lfloor \frac {sq_i}{k(r-2)+1}\rfloor\le d-\lfloor \frac {sq_i}{k(r-1)}\rfloor$. Thus, comparing the inequality $\chi(KG^r_{n,k+l})-1-\lfloor \frac {sq_i}{k(r-1)}\rfloor \ge \chi(KG^r_{Y,q_{i+1}})$ with the inequality (\ref{eq56}), we get

$$\Big\lfloor\frac {sq_i}{k(r-1)}\Big\rfloor+1<  \frac{(2r+1)l^{\beta-1}q_i}{r-1},$$
which contradicts the definition (\ref{eq52}).
\end{proof}

Lemma \ref{lem1} tells us that if there exists a proper $d$-coloring of $KG_{n,k}^r(p)$, then there are subsets $A_1,\ldots, A_r$ as in the lemma, such that the induced subgraph of $KG_{n,k}^r(p)$ on these subsets has no edge. In what follows we calculate the probabilities of these events.

Assume that the first possibility from Lemma \ref{lem1} for a certain $i$ takes place. The probability of the corresponding event for the random hypergraph is at most
\begin{equation}\label{eq8}{n\choose q_i}^r{{q_i\choose k}\choose z_i}^r(1-p)^{z_i^r}.\end{equation}
If the second possibility from Lemma \ref{lem1} for a certain $i$ takes place, then the probability of the corresponding event for the random hypergraph is at most
\begin{equation}\label{eq9}{n\choose q_{i}}{n\choose q_{i+1}}^{r-1}{{q_{i}\choose k}\choose t_i}{{q_{i+1}\choose k}\choose z_{i+1}}^{r-1}(1-p)^{t_iz_{i+1}^{r-1}}\le n^{rq_{i+1}}{{q_{i+1}\choose k}\choose z_{i+1}}^r(1-p)^{t_iz_{i+1}^{r-1}}.\end{equation}

Typically, the last expression in (\ref{eq9}) is much bigger than that in (\ref{eq8}). Note that the total number of events is $2u+1\le 2\log n$. Therefore, if each has probability less than $\frac 1 n$, say, then we a.a.s. have $\chi (KG_{n,k}^r(p))\ge d+1$. We remark that this condition on the probability of a single event is by no means restrictive, since we are manipulating with expressions of much higher order of growth. The following analogue of Lemma \ref{coro1} with  general bounds on the chromatic number of random Kneser hypergraphs is proven by analogous calculations:

\begin{lem}\label{lem2}
   For $n\ge (k+l)r$ we a.a.s. have $\chi(KG^r_{n,k}(p))\ge d+1:=\chi(KG^r_{n,k+l})$ if for each $i =0,\ldots, u$ we have
\begin{align}\label{eq10} &3r\bigl(q_i\log n +z_i\log (2sq_i)\bigr)-pz_i^r\to -\infty \ \ \ \ \ \ \ \text{and}\\
\label{eq11} &3r\bigl(q_{i+1}\log n +z_{i+1}\log (2sq_{i+1})\bigr)-pt_iz_{i+1}^{r-1}\to -\infty.\end{align}
\end{lem}

In what follows, we assume that $p>0$ is fixed and that $k,l\ge 2$. We have $\log q_i = \Omega(\log(2sq_{i+1})) $ for any $i\ge 0$. Then the inequalities (\ref{eq10}) and (\ref{eq11}) follow from

\begin{equation}\label{eq13} z_i^{r-1}\gg \log q_i, \ \ \ z_i^r\gg q_i\log n,  \ \ \ t_iz_{i+1}^{r-2}\gg \log q_{i}, \ \ \ t_iz_{i+1}^{r-1}\gg q_{i+1}\log n.\end{equation}
We have $z_{i+m}/z_i= \Omega(q_{i+m}/q_i)$ for any $r\ge 2, k\ge 2$ and $m\in[u-i]$. Therefore, for $r\ge 3$ it is sufficient to verify the inequalities (\ref{eq13}) for $i=0$. For $r=2$ it is sufficient to verify the first, second and fourth inequality from (\ref{eq13}) for $i=0$ and the inequality $t_0\ge \log n$.\vskip+0.1cm

For $r=2,3$ we have $s= O(l^{1/2})$, which implies $z_i =O\Big(\frac{{q_i\choose k}}{l^{1/2}q_i}\Big)\gg \log q_i$. Therefore, the first inequality from (\ref{eq13}) is satisfied for $r=2,3,$ and \eqref{eq13} reduces to the following:
\begin{equation}\label{eq14} z_0^{r}\gg (k+l)\log n,   \ \ \ t_iz_{i+1}^{r-2}\gg \log q_i, \ \ \ t_0z_{1}^{r-1}\gg (k+l)\log n,\end{equation}
where for $r= 3$ the second inequality is again automatically satisfied.
 Let $r=2$. Looking at the definitions (\ref{eq52}) and using the fact that ${k+l\choose k}=\Omega((k+l)^2)$, it is clear that the second condition is the most restrictive. The following inequality is sufficient to satisfy (\ref{eq14}) and implies both (\ref{eq010}) and (\ref{eq0101}):
 \begin{equation}\label{eq15}{k+l\choose k}\gg n\log n.\end{equation}
 For $r=3$, replacing the $t_0$ factor with $1$, we conclude that (\ref{eq14}) follows from
\begin{equation}\label{eq16} {k+l\choose k}^3\gg (k+l)^4l^{3/2} \log n, \ \ \ {2(k+l^{3/2})\choose k}^2\gg (k+l)^3l\log n.\end{equation}
For fixed $l$ the first condition is clearly more restrictive, and we get that (\ref{eq16}) holds for $k\gg \log^{1/(3l-4)} n$. For fixed $k$ the first inequality is more restrictive again, and we get that it holds for $l\gg \log^{1/(3k-11/2)} n$. This gives the inequalities (\ref{eq011}) and (\ref{eq0111}).\\

For $r>3$ we have to assume $l\ge 3$ in order to get any good lower bound on $z_i$: for $l=2$ we have $z_0=1$. But for $l\ge 3$ we again have $z_i\gg \log q_i$, so it is again enough to verify the first and third inequalities in (\ref{eq14}). Replacing $t_0$ with $1$, we get that (\ref{eq14}) is implied by
\begin{align}\label{eq17} {k+l\choose k}^{r}\gg & \ k^{r}(k+l)^{r+1}l^{r/(r-1)} \log n\ \ \ \ \text{and}\\
\label{eq171} {2(k+l^{r/(r-1)})\choose k}^{r-1}\gg &\ k^{r-1}(k+l)\big(k+l^{\frac{r}{r-1}}\big)^{r-1}l\log n.\end{align}
 Similarly to the case $r=3$, for both fixed $l$ and fixed $k$ \eqref{eq17} is more restrictive. For fixed $l$ we get that (\ref{eq17}) holds for $k\gg \log^{\frac 1{r(l-2)-1}} n$. For fixed $k$ it holds for $l\gg \log^{\frac 1{r(k-1)-\frac{2r-1}{r-1}}} n$. This implies (\ref{eq012}) and (\ref{eq0121}).

\section{Simple lower bounds}\label{sec50}

In this section we present simple upper bounds for $\chi(KG^r_{n,k}(p))$ and compare them with the results of Theorem~\ref{thm0}. If there exist a set $A\subset [n]$ of size $rk+l$, such that  $KG^r(n,k)(p)|_A$ is an empty graph, then, coloring $A$ into color 0 and the rest as in the standard coloring of $KG^r(n,k)$, we get that $\chi(KG^r_{n,k}(p))\le \chi(KG^r_{n,k})- \lfloor l/(r-1)\rfloor.$ To estimate the probability of having such $A$, we find $n$ sets of size $l+2k$ in $[n]$, which have pairwise intersections of size at most 1, and, and calculate the probability that one of those becomes empty. Note that the events for different sets are independent. The probability is
$$\Big(1-(1-p)^{\prod_{i=1}^r{l+ik\choose k}}\Big)^n\le e^{-n(1-p)^{\prod_{i=1}^r{l+ik\choose k}}}.$$
Therefore, if \begin{equation}\label{eq666} n(1-p)^{\prod_{i=1}^r{l+ik\choose k}}\to \infty,\end{equation} then a.a.s. there exists such a set. If $p,r, k$ are fixed, then this condition is satisfied if for sufficiently large constant we have $e^{cl^{rk}}=o(n)$, which implies that we can take $l = \Omega (\log ^{\frac 1{rk}} n)$. This shows that bounds (\ref{eq0111}), (\ref{eq0121}) are essentially tight: the difference between the lower and the upper bounds are in the degree of the logarithm.

If $p,r,$ and $l$ are fixed, then the situation is more interesting. The condition (\ref{eq666}) is satisfied if $e^{c^{r^2k}}=o(n)$, which could be fulfilled for $k=\Omega(\log\log n)$. This is very different from the bounds (\ref{eq011}), (\ref{eq012}). Of course, in the graph case ($r=2$) the gap between the upper and lower bounds is even bigger.


\section{Discussion}\label{sec5}
In \cite{Kup} Kupavskii asked whether it is true that for some $k = k(n)$ a.a.s. we have $\chi(KG_{n,k}(1/2))= \chi(KG_{n,k})$. This question remains wide open for all meaningful values of $k$ (by that we mean that $n-2k\to \infty$), with current methods not allowing to attack it. We ask a similar question for Kneser hypergraphs. This may be easier to show in the hypergraph case. Indeed, when the bound (\ref{eq04}) is applicable, then for sufficiently large $r$ and most $n$ the difference between the chromatic number of $KG_{n,k}^r$ is guaranteed to be at most $1$.

The huge difference in the bounds between the cases $r=2$ and $r\ge 3$ asks for some further exploration. What is the correct order of growth of $k$ needed to guarantee that the chromatic number of a Kneser graph a.a.s. drops by an additive term only when passing to a random subgraph? We conjecture that the following should be true:

\begin{gypo} For any fixed $p>0$ we have $\chi(KG_{n,k}(p))\ge \chi(KG_{n,k}) -4$ for $k\gg \log n$.\end{gypo}

So far most of the research in this direction was concerned with lower bounds. What could one say concerning the upper bounds, or, stated in a more convenient way, lower bounds for the expression $\chi(KG_{n,k}^r)-\chi(KG_{n,k}^r(p))$? In the previous section we showed that for fixed $k$ and $r\ge 3$ we can obtain lower bounds similar to the upper bounds given by (\ref{eq0111}), (\ref{eq0121}). The case of fixed $k$ and $r=2$ seems to be troubling again, as the lower bounds for $\chi(KG_{n,k}^r)-\chi(KG_{n,k}^r(p))$ that are in sight are logarithmic, while the upper ones, provided by (\ref{eq0101}), are polynomial. We also have a huge gap between the upper and  lower bounds on $k$, for which the difference between  $\chi(KG_{n,k}^r)$ and $\chi(KG_{n,k}^r(p)$ is at most a fixed constant $l$, as we show in the previous section.

Finally, we remark that the method from Theorem \ref{thm1} may be applied to the following class of Kneser hypergraphs of set families. Assume that $\ff\subset {[n]\choose k+l}$ is a family of $k+l$-element sets. We form a family $\mathcal H$ of all $k$-element subsets, contained in at least one set from $\ff$ (this is the so-called \textit{$k$-th shadow} of $\ff$). Denote by $KG^r(\mathcal H)$ the $r$-uniform Kneser hypergraph on $\mathcal H$. Then  equation (\ref{eq2}) tells us that  $\chi(KG^r(\mathcal H)(p))\ge d+1:=\chi(KG^r(\ff))$ with probability at least
$$1-|\mathcal H|{{k+l\choose k}\choose \Big\lceil \frac{{k+l\choose k}}d\Big\rceil}^r(1-p)^{\big\lceil \frac{{k+l\choose k}}d\big\rceil^r}.$$
The following question seems to be worth exploring: are there any interesting classes of graphs or hypergraphs, for which the topological bounds (as the ones proven in \cite{Kup} and \cite{AH}) work, while the present combinatorial approach fails?\vskip+0.1cm

\textbf{Note: } New results on the subject, including the the resolution of a slightly weaker version of the conjecture above, are going to appear in \cite{KK}.
\begin{small}

\end{small}
\end{document}